\documentclass[12pt]{article}

\usepackage{authblk}
\usepackage[utf8]{inputenc}
\usepackage{textcomp}
\usepackage[T1]{fontenc}

\usepackage{amssymb}
\usepackage{amsmath}
\usepackage{latexsym}
\usepackage{amsthm}

\usepackage{mathtools}
\usepackage{enumerate}
\usepackage{verbatim}
\usepackage{graphicx}

\usepackage{caption}
\usepackage{subcaption}

\usepackage{comment}
\usepackage{color}

\usepackage{enumitem}

\usepackage{parskip}

\usepackage{tikz}
   \usetikzlibrary{calc}

\newcommand{\overbow}[1]{
   \tikz [baseline = (N.base), every node/.style={}] {
      \node [inner sep = 0pt] (N) {$#1$};
      \draw [line width = 0.4pt] plot [smooth, tension=1.3] coordinates {
         ($(N.north west) + (0.1ex,0)$)
         ($(N.north)      + (0,0.5ex)$)
         ($(N.north east) + (0,0)$)
      };
   }
}

\allowdisplaybreaks

\makeatletter

\newcommand{\C}{\mathbb C}
\newcommand{\R}{\mathbb R}
\newcommand{\Z}{\mathbb Z}
\newcommand{\N}{\mathbb N}

\newcommand{\re}{\mathrm{Re}}
\newcommand{\im}{\mathrm{Im}}

\newcommand{\e}{\varepsilon}

\newtheorem{thm}{Theorem}[section]
\newtheorem{lem}[thm]{Lemma}
\newtheorem{prop}[thm]{Proposition}

\newtheorem{dfn}[thm]{Definition}

\newtheorem{cor}[thm]{Corollary}
\newtheorem{rk}[thm]{Remark}

\numberwithin{equation}{section}

\title{Hyers-Ulam stability of parabolic M\"obius difference equation}

\author[]{Young Woo Nam}

\affil[]{\small Mathematics Section,
      College of Science and Technology,
      Hongik University, 339--701 
      Sejong, Korea} 

\begin{document}

\date{}
\maketitle

\begin{abstract}
The linear fractional map $ g(z) = \frac{az+ b}{cz + d} $ with complex number coefficients on the Riemann sphere where $ ad-bc = 1 $ and $ a+d = \pm 2 $ is called {\em parabolic} M\"obius map. Let $ \{ b_n \}_{n \in \mathbb{N}_0} $ be the solution of the parabolic M\"obius difference equation $ b_{n+1} = g(b_n) $ for every $ n \in \mathbb{N}_0 $. Then the sequence $ \{ b_n \}_{n \in \mathbb{N}_0} $ has no Hyers-Ulam stability. 
\end{abstract}

\section{Introduction}

In 1940, Ulam suggested a problem about the stability of  approximate homomorphism between metric groups \cite{ulam}. In detail, if $ f $ is a map from the metric group to itself and it satisfies that
$$ d(f(xy),\ f(x)f(y)) < \e $$
for all $ x,y $ in the given group $ G $, then does the homomorphism $ h $ exist such that $ d(h(x), f(x)) < \delta $ for all $ x \in G $? 
Hyers gave an affirmative answer \cite{hyers} for Cauchy's additive equation in Banach space. Hyers-Ulam stability has been searched in the field of functional equations and differential equations for decades. More recently, this stability has been considered for difference equations. For instance, see \cite{BPX,jung1,jungnam,popa,XB}. 
\\ \smallskip
Denote the set of natural number by $ \N $ and denote the set $ \N \cup \{0\} $ by $ \N_0 $. Let $ \{ b_n \}_{n \in \N_0} $ be the sequence determined by the difference equation
\begin{align} \label{eq-sequence b-n}
b_{n+1} = F(n,b_n)
\end{align}
where $ F $ is the map from $ \N_0 \times \C $ into $ \C $ with an initial point $ b_0 \in \C $ for $ n \in \N_0 $. If the map $  F(n, \cdot) $ is independent of $ n $, then we use the notion $  F(b_n) $ instead of $ F(n,b_n) $. 

\bigskip
\begin{dfn} \label{def-hyers ulam stability}
Let $ \{a_n\}_{n \in \N_0} $ be the complex valued sequence which satisfies the inequality
\begin{align} \label{eq-sequence a-n}
| a_{n+1} - F(n, a_n) | \leq \e
\end{align} 
for a given $ \e > 0 $ and for all $ n \in \N_0 $ where $ | \cdot | $ is the absolute value of complex number. For every sequence $ \{a_n\}_{n \in \N_0} $ satisfying \eqref{eq-sequence a-n} if there exists a sequence $ \{b_n\}_{n \in \N_0} $ satisfying \eqref{eq-sequence b-n} for each $ n \in \N_0 $ and $ | a_n - b_n | \leq K(\e) $ for all $ n \in \N $ where $ K(\e) \rightarrow 0 $ as $ \e \rightarrow 0 $, then the difference equation \eqref{eq-sequence b-n} is called that it has {\em Hyers-Ulam stability}. 
\end{dfn}
%
%
%
%
%
\subsection*{Parabolic M\"obius map}
The linear fractional map $ z \mapsto \frac{az+b}{cz+d} $ on the Riemann sphere, $ \hat{\C} = \C \cup \{ \infty \} $ is called M\"obius map where $ a,b,c $ and $ d \in \C $ and $ ad-bc \neq 0 $. Every M\"obius map has the matrix representation $\left(\begin{smallmatrix}a&b\\c&d\end{smallmatrix}\right)$. Since the map $ z \mapsto \frac{az+b}{cz+d} $ is the same as $ z \mapsto \frac{qaz+qb}{qcz+qd} $, the matrix $ \left(\begin{smallmatrix}a&b\\c&d\end{smallmatrix}\right) $ and $ \left(\begin{smallmatrix}qa&qb\\qc&qd\end{smallmatrix}\right) $ represent the same map. Thus we may assume that the determinant of matrix representation, namely, $ ad-bc $ is one. Denote the trace of the matrix representation is $ a+d $ under the condition $ ad-bc = 1 $. 
The map $ g(z) = \frac{az+b}{cz+d} $ where $ ad-bc=1 $ and $ a+d = \pm 2 $ is called {\em parabolic} M\"obius map. If $ c \neq 0 $, then we define $ g(\infty) = \frac{a}{c} $ and $ g\left( -\frac{d}{c} \right) = \infty $.

\subsection*{Non stability}
Let $ g(z) = \frac{az+b}{cz+d} $ be the linear fractional map for $ c \neq 0 $. Thus $ g\left( -\frac{d}{c} \right) = \infty $. However, the meaning of Hyers-Ulam stability is not clear to the inequality $ | a_{n+1} - g(a_n) | \leq \e $ where $ a_n = -\frac{d}{c} $ for some $ n \in \N $. Then we consider that some region in $ \C $ which contains the whole sequence $ \{a_n\}_{n\in \N_0} $ and is disjoint from $ \{ g^{-n}(\infty) \ |\ n \in \N \} $ for Hyers-Ulam stability. This region is dependent on $ -\frac{d}{c} $ which is determined by the map. Then the common region of all $ g $ is not considered unless some additional  specific condition guarantees the existence of common region for all sequences from different maps. Similarly, the number $ K(\e) $ depends on each sequence. The non stability in the sense of Hyers-Ulam is discussed in \cite{bbc,BPRX}. For difference equations non stability means that 
\smallskip \\
"At any given $ \e>0 $, there is no region which contains the sequence $ \{a_n\}_{n\in \N_0} $ satisfying the definition of Hyers-Ulam stability".
\smallskip \\
The initial point $ a_0 $ may be chosen arbitrarily on some dense subset of $  \C $ for non stability in the Hyers-Ulam.

\subsection*{Main content}
In Section 2 and Section 3, invariant circles and the extended line (defined later) are constructed and the convergence of $ g^n(z) $ to the unique fixed point is proved as $ n \rightarrow \pm \infty $ along invariant circle. In Section 4, we show the non-stability of the sequence from parabolic M\"obius difference equation 
 in the sense of Hyers-Ulam. In particular, a periodic sequence $ \{a_n\}_{n \in \N_0} $ which satisfies \eqref{eq-sequence a-n} 
is constructed and it is compared with any sequence $ \{b_n\}_{n \in \N_0} $ defined by the equation \eqref{eq-sequence b-n}. In Section 5 and Section 6, we show the non-stability of real parabolic M\"obius difference equation defined on the extended real line. Without any finite invariant circle, the proof requires different calculation for the real parabolic M\"obius map. 

\section{Horocycles}
Let $ g(z) = \frac{az + b}{cz + d} $ be the parabolic M\"obius map where $ ad-bc=1 $, $ c \neq 0 $. The extended line is defined as the union of the straight line and $ \{ \infty \} $ in the Riemann sphere. We define horocycle at the fixed point of $ g $ in Definition \ref{def-horocycle} and show that each of these is invariant under $ g $. 

\bigskip
\begin{lem}
Let $ g(z) = \frac{az + b}{cz + d} $ be M\"obius map where $ ad-bc=1 $, $ c \neq 0 $. If $ g $ is the parabolic M\"obius map, that is, $ a+d = \pm 2 $, then $ g $ has the unique fixed point, say $ \alpha $, and $ \alpha = \frac{a-d}{2c} $.
\end{lem}

\begin{proof}
The fixed point of $ g $ satisfies the equation $ cz^2 - (a-d)z - b = 0 .$ The unique solution of the above quadratic equation is $ \frac{a-d}{2c} $. 
\end{proof}

\bigskip
\begin{lem}  \label{lem-invariant extended line}
Let $ g(z) = \frac{az + b}{cz + d} $ be the parabolic M\"obius map for $ c \neq 0 $. Then the extended line which contains $ \frac{a}{c} $, $ -\frac{d}{c} $ and $ \infty $, say $ L_{\infty} $, is invariant under $ g $. 
\end{lem}
\begin{proof}
The image of circle or line under M\"obius map is circle or line. The extended line $ L_{\infty} $ contains the fixed point, $ \alpha = \frac{a-d}{2c} $ because $ \alpha = \frac{1}{2} \left( \frac{a}{c} + \left(-\frac{d}{c} \right) \right) $. Observe that $ g(\infty) = \frac{a}{c} $ and $ g \left(-\frac{d}{c}\right) = \infty $. Thus since $ g(L_{\infty}) $ contains $ \infty $, $ \alpha $ and $ \frac{a}{c} $, $ g(L_{\infty}) $ is the extended line and it is the same as $ L_{\infty} $. 
\end{proof} 


\bigskip
\begin{dfn} \label{def-horocycle}
Horocycle at the fixed point of parabolic M\"obius map is defined as follows
\begin{enumerate}
\item the extended line $ L_{\infty} $ which contains $ \frac{a}{c} $ and $ -\frac{d}{c} $ or 
\item Every circle which intersects $ L_{\infty} $ at the fixed point of $ g $ tangentially. 
\end{enumerate}
\end{dfn}
By the above definition, the center of each horocycle is contained in the straight line which contains $ \alpha $ and is perpendicular to $ L_{\infty} $. Define the straight line as follows
\begin{align} \label{eq-centers of horocycles}
\ell = \left\{ z \colon \left| z + \frac{d}{c} \right| = \left| z - \frac{a}{c} \right| \right\} .
\end{align}
Then every horocycle is the circle $ S_p = \{ z \colon |z-p| = | \alpha - p| \} $ where $ p $ is a point in $ \ell $. 
\medskip
\begin{rk}
The horocycle in the above definition is slightly different from the usual definition in hyperbolic geometry \cite{beardon} as the set in the complex plane. Our definition of horocycle contains the extended line for convenience. 
\end{rk}
\bigskip
\begin{prop}  \label{prop-horocycle is invariant}
Let $ g(z) = \frac{az + b}{cz + d} $ be the M\"obius map where $ ad-bc=1 $ and $ c \neq 0 $. Suppose that $ g $ is parabolic, that is, $ a+d = \pm 2 $. Then every horocycle at $ \alpha $ of $ g $ is invariant under $ g $, that is, $ z $ satisfies that $ |z-p| = |p-\alpha| $ if and only if $ |g(z)-p| = |\alpha - p| $. 
\end{prop}
The proof of the above proposition requires the following lemmas.
\begin{figure}
    \centering
    \includegraphics[scale=0.85]{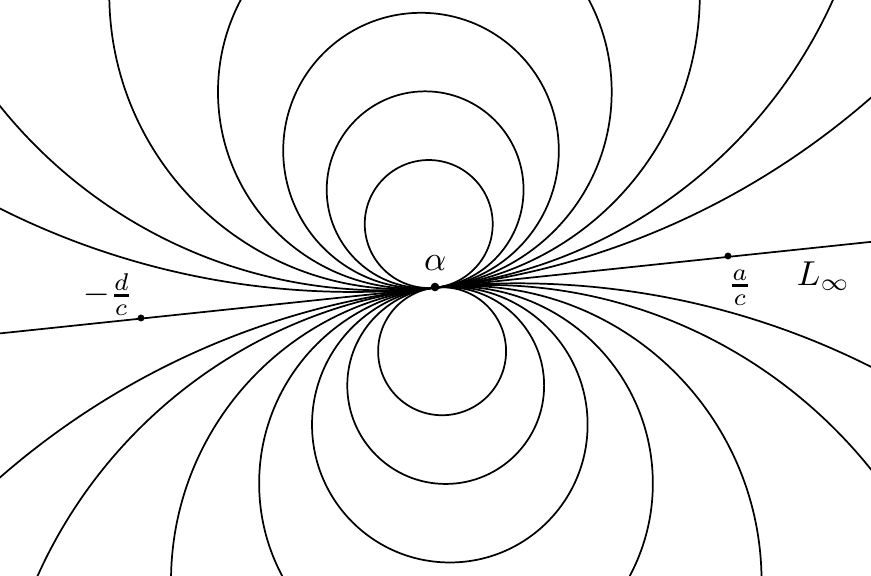}   
    \caption{Horocycles at the fixed point of parabolic M\"obius map}
\end{figure}
\bigskip
\begin{lem} \label{lem-equation for horocycle 1}
Let $ g(z) = \frac{az + b}{cz + d} $ be the parabolic M\"obius map for $ c \neq 0 $. Then 
$$ \overline{cp -a} = -(cp+d) $$ where $ p \in \ell $ defined in \eqref{eq-centers of horocycles}. 
\end{lem}

\begin{proof}
The fact that $ p \in \ell $ implies that $ \left| p + \frac{d}{c} \right| = \left| p - \frac{a}{c} \right| $, that is, $ | cp-a | = | cp+d | $. However, the difference of these complex numbers is
$$ cp-a -(cp+d) = -a-d = \pm 2  $$
the non zero real number. Then $ \im\,(cp-a) = \im\,(cp+d) $. However, since $ | cp-a | = | cp+d | $, we obtain that $ \re\,(cp-a) = -\re\,(cp+d) $. Hence, $ \overline{cp -a} = -(cp+d) $. 
\end{proof}

\medskip

\begin{cor} \label{cor-equation for horocycle 01}
Let $ g(z) = \frac{az + b}{cz + d} $ be the parabolic M\"obius map for $ c \neq 0 $. Then $ c(p- \alpha) $ is a purely imaginary number. 
\end{cor}

\begin{proof}
The $ \alpha = \frac{a-d}{2c} $ is the fixed point of $ g $. Thus
$$ c(p- \alpha) = cp -c\alpha = cp - \frac{a-d}{2} = \frac{1}{2} \left[\,\! (cp-a) + (cp+d) \right] . $$
However, Lemma \ref{lem-equation for horocycle 1} implies that $ \re\,(cp-a) = -\re\,(cp+d) $. Then the sum $ cp-a + cp+d $ is a purely imaginary number. Hence, so is $ c(p- \alpha) $. 
\end{proof}
\medskip
\begin{cor} \label{cor-intermediate calculation 1}
Let $ g(z) = \frac{az + b}{cz + d} $ be the parabolic M\"obius map for $ c \neq 0 $. Then the equation 
\begin{align*}
\frac{d}{c} + \frac{\overline{cp-a}}{c} = -p 
\end{align*}
holds. 
\end{cor}

\begin{proof}
Lemma \ref{lem-equation for horocycle 1} implies that 
$$ \frac{d}{c} + \frac{\overline{cp-a}}{c} = \frac{d}{c} - \frac{cp+d}{c} = -p . $$
\end{proof}

\begin{lem} \label{lem-pythagorean rule for complex numbers}
Let $ g(z) = \frac{az + b}{cz + d} $ be the parabolic M\"obius map for $ c \neq 0 $. Let $ \alpha $ be the fixed point of $ g $ and $ p $ is contained in $ \ell $. Then the equation 
\begin{align*}
|cp - a|^2 = | c(p - \alpha)|^2 + 1 
\end{align*}
holds. 
\end{lem}

\begin{proof}
The fixed point $ \alpha $ is $ \frac{a-d}{2c} $ and $ a+d = \pm 2 $. Thus
\begin{align} \label{eq- c alpha-a}
c\alpha - a =  \frac{c(a-d)}{2c} - a = - \frac{a+d}{2} = \pm 1
\end{align}
which is the real number. Corollary \ref{cor-equation for horocycle 01} implies that $ c(p-\alpha) $ is the purely imaginary number. Since $ cp-a = cp - c\alpha + c\alpha - a $, the Pythagorean theorem holds
\begin{align*}
|cp - a|^2 = |cp - c\alpha|^2 + |c\alpha - a|^2 \ .
\end{align*}
Hence, the equation \eqref{eq- c alpha-a} implies that 
$ |cp - a|^2 = |cp - c\alpha|^2 + 1 $. 
\end{proof}

\bigskip
\begin{lem}  \label{lem-circle equation}
Let $ A $ and $ B $ be the complex numbers satisfying that $ |A| \neq |B| $. Then $ \left| \frac{1}{z} + A \right| = |B| $ if and only if 
\begin{align*}
\left| z + \frac{\overline{A}}{|A|^2 - |B|^2} \right| = \frac{|B|}{\big| |A|^2 - |B|^2 \big|}
\end{align*}
\end{lem}

\begin{proof}
The following equivalent equation completes the proof
\begin{align*}
& \;\! \left| \frac{1}{z} + A \right| = |B| \\[0.2em]
\Longleftrightarrow & \ | 1+ Az|^2 = |Bz|^2 \\[0.2em]
\Longleftrightarrow & \ 1+ Az + \overline{Az} + |Az|^2 = |Bz|^2 \\[0.2em]
\Longleftrightarrow & \ \big( |A|^2 - |B|^2 \big) |z|^2 + Az + \overline{Az} = -1 \\
\Longleftrightarrow & \ |z|^2 + \frac{Az}{|A|^2 - |B|^2} + \frac{\overline{Az}}{|A|^2 - |B|^2} = - \frac{1}{|A|^2 - |B|^2} \\
\Longleftrightarrow & \ \left| z + \frac{\overline{A}}{|A|^2 - |B|^2} \right|^2 = \frac{|A|^2}{(|A|^2 - |B|^2)^2} - \frac{1}{|A|^2 - |B|^2} \\
\Longleftrightarrow & \ \left| z + \frac{\overline{A}}{|A|^2 - |B|^2} \right| = \frac{|B|}{\big| |A|^2 - |B|^2 \big|} .
\end{align*}
\end{proof}


\begin{proof}[Proof of Proposition \ref{prop-horocycle is invariant}]
Suppose that $ |g(z)-p| = |p-\alpha| $. The following equations are equivalent 
\begin{align}  
& \;\! 
\left| \frac{az+b}{cz+d} - p \,\!\right| = |p - \alpha|
 \nonumber \\ 
\Longleftrightarrow & \left| \frac{a}{c} - \frac{1}{c(cz+d)} -p \right| = |p - \alpha| \nonumber \\ 
\Longleftrightarrow & \left| \frac{1}{cz + d} + cp-a \right| = |c||p-\alpha| .\nonumber 
\\
%
%
\intertext{Lemma \ref{lem-circle equation}, Lemma \ref{lem-pythagorean rule for complex numbers} and Corollary \ref{cor-intermediate calculation 1} implies that the equation as follows 
}
%
%
\Longleftrightarrow 
& \left| cz+d + \frac{\overline{cp-a}}{|cp-a|^2 - |c(p-\alpha)|^2} \right| = \frac{|c(p-\alpha)|}{|cp-a|^2 - |c(p-\alpha)|^2} \nonumber \\[0.3em]
\Longleftrightarrow & \,|cz+d + \overline{cp-a}| = |c(p-\alpha)|  \qquad \textrm{by Lemma \ref{lem-pythagorean rule for complex numbers}} \nonumber \\[0.2em]
\Longleftrightarrow & \left| z + \frac{d}{c}+\frac{\overline{cp-a}}{c} \right| = |p-\alpha| \nonumber \\
\Longleftrightarrow & \,|z-p|=|p-\alpha|  \hspace{3.1cm} \textrm{by Corollary \ref{cor-intermediate calculation 1}} \nonumber . 
\end{align}
Then we obtain that the equation $ |g(z)-p| = |p-\alpha| $ is satisfied if and only if the equation $ |z-p| = |p-\alpha| $ is so. Moreover, the extended line $ L_{\infty} $ is invariant under $ g $ by Lemma \ref{lem-invariant extended line}. Hence, any point $ q $ is contained in a horocycle if and only if $ g(q) $ is contained in the same horocycle. 
\end{proof}

\section{Conjugation between parabolic M\"obius map and translation}
The map $ h(z) = \frac{1}{c(z-\alpha)} $ is the conjugation between parabolic M\"obius map and the translation $ z \mapsto z+1 $, that is, $ h \circ g(z) = h(z) +1 $. Let the set $ \{ g(z), g^2(z), \ldots ,g^n(z), \ldots \} $ be the (forward) orbit of $ z $ under $ g $. We show that the forward orbit of any point in $ \C $ under $ g $ converges to the fixed point. Moreover, in this section we show that each point of the orbit, $ g^n(z) $ is arranged with either clockwise or anticlockwise direction along a horocycle as $ n $ increases. 

\bigskip
\begin{lem} \label{lem-conjugation h between g and translation}
Let $ g(z) = \frac{az + b}{cz + d} $ be the parabolic M\"obius map for $ ad-bc=1 $ and $ c \neq 0 $. Let $ h $ be the map, $ h(z) = \frac{1}{c(z-\alpha)} $ where $ \alpha $ is the unique fixed point of $ g $. Then the map $ h \circ g \circ h^{-1} $ is the translation $ z \mapsto z+1 $ where $ a+d=2 $ or $ h \circ g \circ h^{-1} $ is the translation $ z \mapsto z-1 $ where $ a+d=-2 $. 
\end{lem}

\begin{proof}
Recall that any composition of two M\"obius map is also M\"obius map. Thus if a M\"obius map has $ \infty $ as a fixed point, then it is a linear map. Denote $ h \circ g \circ h^{-1} $ by $ f $. The straightforward calculation implies that $ h^{-1}(w) = \alpha+\frac{1}{cw} $. Recall that $ \alpha = \frac{a-d}{2c} $. 
Suppose that $ a+d=2 $ firstly. The following is an intermediate calculations 
\begin{align} \label{eq-intermediate calculation}
\alpha + \frac{1}{c} = \frac{a-d}{2c} + \frac{1}{c} = \frac{a}{c} \ , \quad 
\alpha - \frac{1}{c} = \frac{a-d}{2c} - \frac{1}{c} = -\frac{d}{c} .
\end{align}
Then we have the equations as follows using \eqref{eq-intermediate calculation}
\begin{align*}
 f(\infty) &= h \circ g \circ h^{-1}(\infty) = h \circ g(\alpha) = h(\alpha) = \infty \\
 f(0) &= h \circ g \circ h^{-1}(0) = h \circ g(\infty) = h \left( \frac{a}{c} \right) = h \left( \alpha + \frac{1}{c} \right) = 1 \\
 f(-1) &= h \circ g \circ h^{-1}(-1) = h \circ g \left(\alpha - \frac{1}{c} \right) = h \left( -\frac{d}{c} \right) = h \left( \infty \right) = 0  .
\end{align*}
Since $ f $ is a M\"obius map, the fact that $ f(\infty) = \infty $, $ f(0) = 1 $ and $ f(-1) = 0 $ implies that $ f(z) = z + 1 $. In the case that $ a+d=-2 $, the proof which is similar to the above one is applicable. Thus the remaining proof is left to the readers. 
\end{proof}

\medskip

\begin{lem} \label{lem-image of horocycle under h}
Let $ g(z) = \frac{az + b}{cz + d} $ be the parabolic M\"obius map for $ ad-bc=1 $ and $ c \neq 0 $. Let $ h $ be the map, $ h(z) = \frac{1}{c(z-\alpha)} $. Then $ h $ maps each horocycle at $ \alpha $ of $ g $ to the extended line parallel to the $ x $-axis in $ \C $. 
\end{lem}

\begin{proof}
Any point in a horocycle at $ \alpha $ of $ g $ satisfies the equation $ |z-p| = | \alpha -p| $ for some $ p \in \ell $. Denote $ h(z) $ by $ w $. Recall that $ h^{-1}(w) = z = \alpha + \frac{1}{cw} $. Thus the following equations are equivalent
\begin{align} \label{eq-number in horocycle}
|z-p| = | \alpha -p| & \Longleftrightarrow \left| \frac{1}{cw}+\alpha-p \right| = | \alpha -p| \nonumber \\
 & \Longleftrightarrow \left| \frac{1}{cw}+\alpha-p \right|^2 = | \alpha -p|^2  \nonumber \\
 & \Longleftrightarrow \frac{1}{|cw|^2} + \frac{\overline{\alpha-p}}{cw} + \frac{\alpha-p}{\overline{cw}} + | \alpha -p|^2 = | \alpha -p|^2  \nonumber \\
 & \Longleftrightarrow \frac{1}{|cw|^2} \left( 1+ \overline{(\alpha-p)cw} + (\alpha-p)cw \right) = 0  \nonumber \\[0.2em]
 & \Longleftrightarrow \overline{(p-\alpha)cw} + (p-\alpha)cw = 1 .
\end{align}
Corollary \ref{cor-intermediate calculation 1} implies that $ c(p-\alpha) $ is a purely imaginary number. Thus the equation 
\begin{align*}
\overline{c(p-\alpha)} = -c(p-\alpha)
\end{align*}
holds. The equation \eqref{eq-number in horocycle} implies that 
\begin{align*}
1 = \overline{c(p-\alpha)}\overline{w} + c(p-\alpha)w = c(p-\alpha) \big[\,\! -\overline{w} + w \big] = c(p-\alpha) (2 \,\im\, w )\,i 
\end{align*}
Hence, the imaginary part of $ w $ is constant as follows 
\begin{align} \label{eq-image of horocycle under h}
\im\, w = \frac{1}{2c(p-\alpha)i} .
\end{align}
It completes the proof.  
\end{proof}
\smallskip
\begin{cor}
The point $ z $ satisfies that $ \left| z + \frac{d}{c} \right| = \left| z - \frac{a}{c} \right| $ if and only if the real part of $ h(z) $ is zero, that is, $ \re\, h(z) = 0 $. 
\end{cor}
\begin{proof}
Denote $ h(z) $ by $ w $. Observe that for any $ z \in \C $ there exists a horocycle at $ \alpha $ which contains $ z $. Thus we may assume that $ \im\,w = \frac{1}{2c(p-\alpha)i} $ by the equation \eqref{eq-image of horocycle under h} in Lemma \ref{lem-image of horocycle under h}. Suppose that $ w = i\,\im\,w $. Then since $ h(z) = \frac{1}{c(z-\alpha)} $, the following equation holds for $ z $  
\begin{align*}
\frac{1}{c(z-\alpha)} = \frac{1}{2c(p-\alpha)} \,.
\end{align*}
Thus we obtain that $ z = 2(p-\alpha) + \alpha $. Recall that $ \alpha = \frac{a-d}{2c} $ and $ a+d = \pm 2 $. Thus 
\begin{align*}
c\alpha + d &= c \,\frac{a-d}{2c} + d = \frac{a+d}{2} = 1 \ \ \textrm{or}  \ -1 \\[0.3em]
c\alpha - a &= c \,\frac{a-d}{2c} - a = -\frac{a+d}{2} = -1 \ \ \textrm{or} \ \ 1 .
\end{align*}
Then both $ c\alpha + d $ and $ c\alpha - a $ are real numbers and the sum of these two numbers is zero. Recall also that Corollary \ref{cor-equation for horocycle 01} implies that $ c(p-\alpha) $ is a purely imaginary number. Thus for any real number $ r $, distance between $ c(p-\alpha) $ and $ r $ is the same as the distance between $ c(p-\alpha) $ and $ -r $. Then we obtain that 
\begin{align*}
\left| z + \frac{d}{c} \right| &= \left| 2(p-\alpha) + \alpha + \frac{d}{c} \right| \\
&= \frac{1}{|c|} \left| 2c(p-\alpha) + c\alpha + d \right| \\
&= \frac{1}{|c|} \left| 2c(p-\alpha) + c\alpha - a \right| \\
&= \left| 2(p-\alpha) + \alpha - \frac{a}{c} \right| \\
&= \left| z - \frac{a}{c} \right|  .
\end{align*}
The map $ h $ is the bijection between each horocycle and the corresponding extended line in Lemma \ref{lem-image of horocycle under h}. The line $ \ell $ defined in \eqref{eq-centers of horocycles}
\begin{align*} 
\ell = \left\{ z \colon \left| z + \frac{d}{c} \right| = \left| z - \frac{a}{c} \right| \right\} 
\end{align*}
is the straight line which goes through $ \alpha $ and the center of every horocycles. Thus the intersection of a single horocycle and $ \ell $ is the set of two points, one of which is $ \alpha $. The other point is the unique point $ z_0 $ satisfying $ h(z_0) = \frac{1}{2c(p-\alpha)} $ which is determined by the equation $ |p-\alpha| = |z_0 - p| $ where $ p \in \ell $. Since $ c(p-\alpha) $ is a purely imaginary number, so is $ h(z_0) $. This completes the proof. 
\end{proof}
Let $ \overbow{qs} $ be the circular arc of which end points are $ q $ and $ s $. The points $ q $ and $ s $ are called the end points of $ \overbow{qs} $. 
\begin{enumerate}[label=\text{(3.\alph*)}]
\item We consider only circular arc in $ \C $ which is homeomorphic to the closed interval.
\item The definition of circular arc is extended to the closed interval, which is a subset of the extended line. 
\item The embedded homeomorphic image of the circular arc in $ \C $ is called {\em arc} and the notation of arc is the same as that of circular arc. 
\item \label{item-condition for circular arc} The arc $ \overbow{z_0 z_1 z_2 \ldots z_{n-1} z_n} $ is defined as the arc $ \overbow{z_0 z_n} $ which contains the points $ z_1, z_2, \ldots , z_{n-1} $ where $ \overbow{z_i z_{i+1}} $ is disjoint from $ \overbow{z_j z_{j+1}} $ only if $ 0 < i+1 < j < n $ for $ n \geq 4 $. 
\end{enumerate}
\bigskip  
\begin{lem}  \label{lem-homeomorphic image of arc}
Let $ h $ be a homeomorphism on the Riemann sphere. Let the points $ z_0,z_1, \ldots , z_n $ in $ \C $ be in the single circular arc $ \overbow{z_0 z_n} $ and $ w_j $ be $ h(z_j) $ in $ \C $ for $ 0 \leq j \leq n $. Suppose that the arc $ \overbow{z_0 z_1 z_2 \ldots z_{n-1} z_n} $ satisfies the condition \ref{item-condition for circular arc}. Then $ h(\overbow{z_0 z_1 z_2 \ldots z_{n-1} z_n}) $ is either $ \overbow{w_0 w_1 w_2 \ldots w_{n-1} w_n} $ or $ \overbow{w_n w_{n-1} \ldots w_{2} w_1 w_0} $. 
\end{lem}
\begin{proof}
Suppose that $ h(\overbow{z_0 z_i z_{j} z_n}) $ is $ \overbow{w_0 w_j w_i w_n} $ where $ 0<i<j<n $ for $ n \geq 4 $. Then by the condition of the arc, $ \overbow{z_0 z_i} $ is disjoint from the arc $ \overbow{z_{j} z_n} $. However, the intersection, $ \overbow{w_0 w_i} \cap  \overbow{w_j w_n} = \overbow{w_j w_i} $ is not empty. It contradicts that $ h $ is the homeomorphism. For the case $ n=3 $, suppose that $ \overbow{z_0 z_1 z_2} $ is mapped  to $ \overbow{w_0 w_1 w_2} $ or $ \overbow{w_2 w_1 w_0} $, that is, $ w_2 $ is an end point of $ h(\overbow{z_0 z_1 z_2}) $. However, $ \overbow{z_0 z_1 z_2} \setminus \{z_2 \} $ is disconnected but $ \overbow{w_0 w_1 w_2} \setminus \{w_2 \} $ or $ \overbow{w_2 w_1 w_0} \setminus \{w_2 \} $ is connected. It contradicts that $ h $ is the homeomorphism. 
\end{proof}
\medskip 
Let $ S_p $ be a horocycle which is contained in $ \C $ as follows 
\begin{align} \label{eq-a horocycle Sp}
S_p = \{ z \colon |z-p| = | \alpha -p| \,\}
\end{align}
where $ p $ satisfies that $ \left| p + \frac{d}{c} \right| = \left| p - \frac{a}{c} \right| $. Let the principal argument of $ \alpha - p $ be the argument between $ -\pi $ to $ \pi $, that is, $ -\pi < \mathrm{Arg}(\alpha - p) \leq \pi $. 
\begin{enumerate}[resume, label=\text{(3.\alph*)}]
\item \label{item-principal argument in horocycle} Let $ z_0,z_1, \ldots , z_n $ be the points contained in $ S_p \setminus \{ \alpha \} $. Let $ \theta_j $ be the argument of $ z_j - p $ where $ \mathrm{Arg}(\alpha - p) < \theta_j < \mathrm{Arg}(\alpha - p) + 2\pi $ for every $ 1 \leq j \leq n $. 
\end{enumerate}
Thus the arc $ \overbow{z_1 z_2 \ldots z_{n-1} z_n} $ satisfies that the condition \ref{item-condition for circular arc} if and only if the arguments of $ z_j $ for $ 1 \leq j \leq n $ satisfies that $ \theta_1 < \theta_2 < \cdots < \theta_n $ or $ \theta_n < \theta_{n-1} < \cdots < \theta_1 $. Similarly, if $ w_1,w_2,\ldots,w_n $ are contained in the line parallel to real axis, that is, $ w_j \in \{ w \colon \,\im\,w = \textrm{const.} \} $ for all $ 1 \leq j \leq n $, then the arc $ \overbow{w_1 w_2 \ldots w_{n-1} w_n} $ satisfies the condition \ref{item-condition for circular arc} if and only if the real part of $ w_j $ for $ 1 \leq j \leq n $ satisfies that $ \re\,w_1 < \re\,w_2 < \cdots < \re\,w_n $ or $ \re\,w_n < \re\,w_{n-1} < \cdots < \re\,w_1 $. Then the above statements and Lemma \ref{lem-homeomorphic image of arc} implies the following lemma. 
%
%
%
\begin{figure}
    \centering
    \includegraphics[scale=0.90]{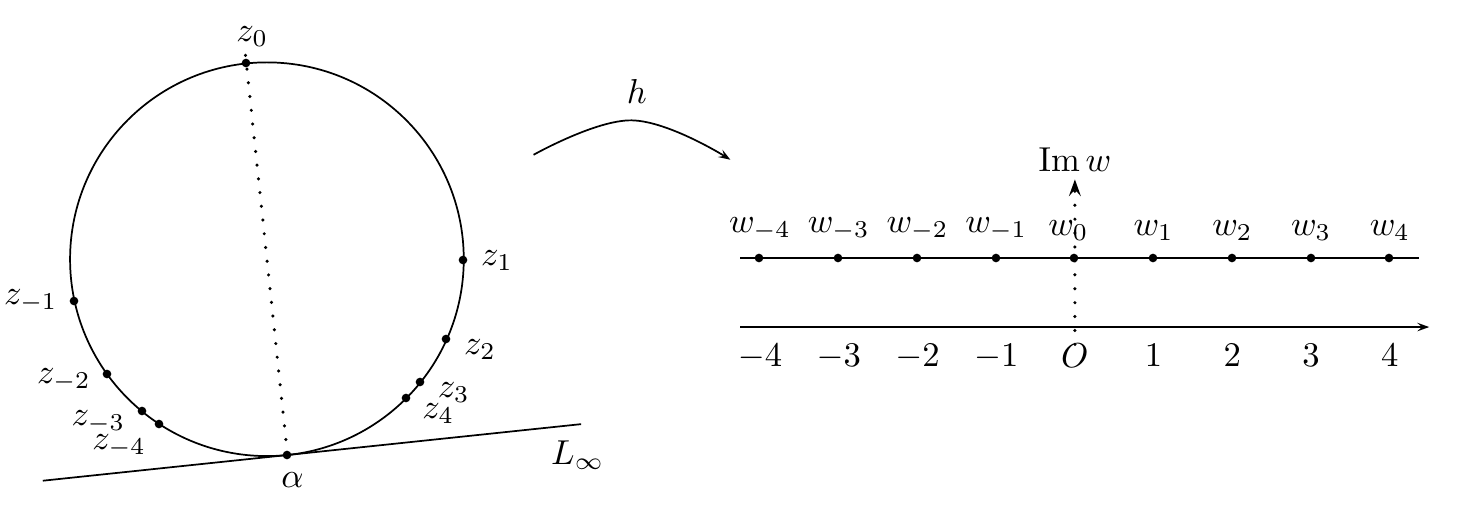}   
    \caption{Image of horocycle under $ h $}
\end{figure}
\bigskip
\begin{lem} \label{lem-order of the argument}
Let $ z_0,z_1, \ldots , z_n $ be the points contained in $ S_p \setminus \{ \alpha \} $ where $ S_p $ is the horocycle at $ \alpha $ in \eqref{eq-a horocycle Sp}. 
Let $ \theta_j $ be the argument of $ z_j - p $ where $ \mathrm{Arg}(\alpha - p) < \theta_j < \mathrm{Arg}(\alpha - p) + 2\pi $ for every $ 1 \leq j \leq n $. Let $ w_1,w_2,\ldots,w_n $ be points in the set $ h(S_p \setminus \{ \alpha \}) $ where $ h(z) = \frac{1}{c(z-\alpha)} $. Then $ \theta_1 < \theta_2 < \cdots < \theta_n $ or\ $ \theta_n < \theta_{n-1} < \cdots < \theta_1 $ if and only if \ $ \re\,w_1 < \re\,w_2 < \cdots < \re\,w_n $. 
\end{lem}

\begin{proof}
The $ h $ maps horocycle at $ \alpha $ of $ g $ to the extended line parallel to the $ x $-axis in $ \hat{\C} $ by Lemma \ref{lem-image of horocycle under h}. Moreover, Lemma \ref{lem-homeomorphic image of arc} and the relation between arc and the argument of $ z_j $ for $ 1 \leq j \leq n $ completes the proof. 
\end{proof}

\medskip
\begin{prop} \label{prop-points with order of arguments}
Let $ S_p $ be the horocycle at $ \alpha $ of $ g $ where $ g(z) = \frac{az+b}{cz+d} $ with $ ad-bc=1 $ and $ c \neq 0 $. For a point $ z_0 \in S_p \setminus \{ \alpha \} $, denote $ g^n(z_0) $ by $ z_n $ for $ n \in \Z $. Then 
$$ \lim_{n \rightarrow \pm \infty} z_n = \alpha . $$
Moreover, either $ \mathrm{Arg}(\alpha - p) < \theta_i < \theta_j < \mathrm{Arg}(\alpha - p) + 2\pi $ where $ i < j $ or $ \mathrm{Arg}(\alpha - p) < \theta_i < \theta_j < \mathrm{Arg}(\alpha - p) + 2\pi $ where $ i > j $ for $ i, j \in \Z $. 
\end{prop}

\begin{proof}
Lemma \ref{lem-conjugation h between g and translation} implies that $ h \circ g \circ h^{-1}(w) = w +1 $ where $ h(z) = \frac{1}{c(z-\alpha)} $. Denote the translation $ w \mapsto w + 1 $ by $ f $. Let $ w_j $ be $ h(z_j) $ for every $ j \in \Z $. Since every point $ z_j $ is in a single horocycle for all $ j \in \Z $ by Proposition \ref{prop-horocycle is invariant}, every $ w_j $ is also in a single straight line $ \{ w \colon \im\,w = \textrm{const.} \} $ for all $ j \in \Z $. Moreover, by the conjugation $ h $ we have that 
\begin{align*}
w_{j+1} &= h(z_{j+1}) = h \circ g(z_j) = h \circ g \circ h^{-1}(w_j) = f(w_j) = w_j + 1 .
\end{align*}
Thus $ \re\,w_i < \re\,w_j $ where $ i < j $ for each $ i,j \in \Z $. Lemma \ref{lem-order of the argument} implies that either $ \theta_i < \theta_j $ for every integer $ i < j $ or  $ \theta_i > \theta_j $ for every integer $ i < j $ where the argument of each point $ z_j $ in the horocycle is defined in \ref{item-principal argument in horocycle}. 
%
By induction, the equation $ w_n = w_0 + n $ holds for $ n \in \Z $. The map $ h $ is the continuous bijection on $ \hat{\C} $ under spherical metric. Then 
\begin{align*}
h(\alpha) &= \infty = \lim_{n \rightarrow \pm \infty} w_n = \lim_{n \rightarrow \pm \infty} h(z_n) = h \left(\lim_{n \rightarrow \pm \infty} z_n \right)
\end{align*}
Hence, $ \alpha = \lim_{n \rightarrow \pm \infty} z_n $. 
\end{proof}

\medskip
\begin{rk}
The orbit of the any point $ z_0 $ in $ \C \setminus \{\alpha\} $, namely, the set $ \{ g^n(z_0) \colon n \in \Z \} $ is contained in a single horocycle at $ \alpha $ of $ g $. Proposition \ref{prop-points with order of arguments} implies that every point, $ g^n(z_0) $ for $ n \in \Z $ are positioned with clockwise or counterclockwise direction. 
\end{rk}

\section{No Hyers-Ulam stability of parabolic M\"obius difference equation}

Let $ \{ b_n\}_{n \in \N_0} $ be the sequence satisfying that $ b_{n+1} = g(b_n) $ for $ n \in \N_0 $ where $ g $ is the parabolic M\"obius map $ g(z) = \frac{az+b}{cz+d} $ for $ ad-bc =1 $. We show that the above sequence $ \{ b_n\}_{n \in \N_0} $ for any initial point $ b_0 \in \C $ has no Hyers-Ulam stability. The result holds for the both cases that $ c \neq 0 $ or $ c=0 $. 

\bigskip
\begin{prop} \label{prop-no hyers-ulam of parabolic mobius}
Let $ g $ be the parabolic M\"obius map which does not fix $ \infty $, that is, $ g(z) = \frac{az+b}{cz+d} $ for $ ad-bc =1 $, $ a+d = \pm 2 $ and $ c \neq 0 $. Let $ \{ b_n\}_{n \in \N_0} $ be the sequence satisfying that $ b_{n+1} = g(b_n) $ for every $ n \in \N_0 $. Then for any initial point $ b_0 \in \C $, the sequence $ \{ b_n\}_{n \in \N_0} $ has no Hyers-Ulam stability. 
\end{prop}

\begin{proof}
For any given $ \e > 0 $ we show that there exists a sequence $ \{ a_n\}_{n \in \N_0} $ which satisfies the following properties. 
\begin{enumerate}
\item $ | a_{n+1} - g(a_n) | \leq \e $ for all $ n \in \N_0 $, 
\item the sequence $ \{ a_n\}_{n \in \N_0} $ is pre-periodic one, and 
\item $ | a_{N_0} - b_{N_0} | \geq 1 + \e $ for some big enough $ N_0 \in \N $. 
\end{enumerate}
Then $ | a_n - b_n | \geq 1 $ for infinitely many $ n \in \N $ due to the periodicity of the sequence $ \{ a_n\}_{n \in \N_0} $. Hence, $ \{ b_n\}_{n \in \N_0} $ has no Hyers-Ulam stability. In the rest of proof, we construct the sequence $ \{ a_n\}_{n \in \N_0} $ satisfying the above properties. For any given $ z_0 \in \C $, $ \displaystyle \lim_{n \rightarrow \pm\infty} g^n(z_0) = \alpha $ by Proposition \ref{prop-points with order of arguments}. Thus for big enough $ N_1 $ the points $ g^n(z_0) $ in the disk, $ B \big(\alpha, \frac{\e}{2} \big) $ for all $ n \geq N_1 $. Thus there exists a point $ q $ in the disk $ B \big(\alpha, \frac{\e}{2} \big) $ which satisfies that   
\smallskip
\begin{enumerate}[label=\text{(4.\alph*)}]
\item $ |q - g^{N_1}(z_0)| \leq \e $, 
\item $ g^{-k}(q) \in B \big(\alpha, \frac{\e}{2} \big) $ for all $ k \in \N_0 $, 
\item the point $ q $ is in the horocycle $ S_p = \{ z \colon |z-p| = | \alpha - p| \} $ in $ \C $ where $ p $ is contained in the line $ \ell $ defined in \eqref{eq-centers of horocycles} for $ | \alpha - p | \geq 1 + 2\e $ and  
\item for some $ N_2 > 0 $, $ g^{N_2}(q) $ is the point in the  intersection of $ \ell \cap S_p \setminus \{ \alpha \} $. \smallskip
\end{enumerate}
Then the following sequence 
\begin{align} \label{eq-preperiodic seq}
\{ z_0, g(z_0), g^2(z_0), \ldots , g^{N_1}(z_0), q, g(q), \ldots , g^{N_2}(q), \ldots , g^{2N_2}(q), q , g(q) , \ldots \}
\end{align}
is pre-periodic sequence with period $ 2N_2 $. For the given sequence $ \{ b_n\}_{n \in \N_0} $ satisfying $ b_{n+1} = g(b_n) $ for all $ n \in \N_0 $, define $ \{ a_n\}_{n \in \N_0} $ as the sequence \eqref{eq-preperiodic seq}, that is, 
$ a_0 = z_0 $, $ a_n = g^n(z_0) $ for $ 1 \leq n \leq N_1 $, $  a_{N_1+1} = q $, $ a_{m+N_1+1} = g^m(q) $ for $ 1 \leq m \leq 2N_2 $ and $ a_{k + 2N_2} = a_{k} $ for every $ k \geq N_1+1 $ where $ z_0 = b_0 $. Then $ \{ a_n\}_{n \in \N_0} $ is the pre-periodic sequence satisfying $ | a_{n+1} - g(a_n) | \leq \e $ for all $ n \in \N_0 $. Moreover, the distance between $ a_{N_2+N_1+1} $ and $ \alpha $ is the diameter of the horocycle $ S_p $. Then 
\begin{align*}
| a_{N_2+N_1+1} - b_{N_2+N_1+1} | \geq | a_{N_2+N_1+1} - \alpha | - | \alpha - b_{N_2+N_1+1}| \geq 1 + 2\e - \e > 1
\end{align*}
By the periodicity of the sequence $ \{ a_n\}_{n \in \N_0} $ for $ n \geq N_1+1 $, we obtain that 
\begin{align*}
| a_{N_2+N_1+1 + 2kN_2} - b_{N_2+N_1+1} | > 1
\end{align*}
for all $ k \in \N_0 $. Hence, the sequence $ \{ b_n\}_{n \in \N_0} $ does not have Hyers-Ulam stability. 
\end{proof}

Assume that the parabolic M\"obius map $ g(z) = \frac{az+b}{cz+d} $ for $ ad-bc =1 $ fixes the infinity, that is, $ c = 0 $.  Since $ a+d = \pm 2 $ and $ ad = 1 $, we obtain that either $ a=d=1 $ or $ a=d=-1 $. Thus the M\"obius map is the translation, $ g(z) = z \pm b $. If $ b=0 $, then the sequence $ \{ b_n\}_{n \in \N_0} $ is the constant sequence and it has no Hyers-Ulam stability. Thus without loss of generality we may assume that $ b \neq 0 $. 

\bigskip
\begin{prop} \label{prop-no hyers-ulam of translation}
Let $ g(z) = z + q $ where $ q $ is a non-zero complex number. Then the sequence $ \{ b_n\}_{n \in \N_0} $ satisfying $ b_{n+1} = g(b_n) $ for any given initial point $ b_0 \in \C $ has no Hyers-Ulam stability. 
\end{prop}

\begin{proof}
For a given $ \e > 0 $, define the sequence $ \{ a_n\}_{n \in \N_0} $ with its elements as follows 
\begin{align*}
a_n = a_0 + n ( \e + q )
\end{align*}
for every $ n \in \N_0 $. Thus 
\begin{align*} 
| a_{n+1} - g(a_n) | = | a_0 + (n+1) ( \e + q ) - [ \,\!a_0 + n ( \e + q )+q ]| = \e .
\end{align*}
By induction each element of the sequence $ \{ b_n\}_{n \in \N_0} $ is of the following form, $ b_n = b_0 + nq $. Then  we obtain that
\begin{align} \label{eq-unbounded difference of translation}
\nonumber | b_n - a_n | &= | b_0 + nq - [a_0 + n ( \e + q )] | = | b_0 - a_0 - n\e | \\
&\geq \big| |b_0-a_0| - n\e \big|  .
\end{align}
However, $ | a_n - b_n | \rightarrow \infty $ as $ n \rightarrow \infty $ by the inequality \eqref{eq-unbounded difference of translation}. Hence, $ \{ b_n\}_{n \in \N_0} $ has no Hyers-Ulam stability. 
\end{proof}

Proposition \ref{prop-no hyers-ulam of parabolic mobius} and Proposition \ref{prop-no hyers-ulam of translation} implies the following theorem. 

\bigskip
\begin{thm}
Let $ g $ be the parabolic M\"obius map, that is, $ g(z) = \frac{az+b}{cz+d} $ for $ ad-bc =1 $ and $ a+d = \pm 2 $. Let $ \{ b_n\}_{n \in \N_0} $ be the sequence satisfying that $ b_{n+1} = g(b_n) $ for every $ n \in \N_0 $. Then for any initial point $ b_0 \in \C $, the sequence $ \{ b_n\}_{n \in \N_0} $ has no Hyers-Ulam stability. 
\end{thm}

\bigskip
\begin{rk}
The proof of Proposition \ref{prop-no hyers-ulam of parabolic mobius} uses the horocycle and pre periodic sequence. In \cite{BW} the periodicity of sine function is used to proving the lack of Hyers-Ulam stability of the linear isometry in general metric space. 
\end{rk}

\section{Real parabolic M\"obius difference equation}
Let $ g $ be the real parabolic M\"obius map, that is, $ g(x) = \frac{ax+b}{cx+d} $ on the extended real line, $ \hat{\R} = \R \cup \{ \infty \} $ where $ a, b, c, d $ are real numbers for $ ad-bc=1 $ and $ a+d =\pm 2 $. The extended real line corresponds the extended line $ L_{\infty} $ of the parabolic M\"obius map on the Riemann sphere. However, $ \hat{\R} $ does not contain any horocycle with finite diameter. In the case that $ c=0 $, the real parabolic M\"obius map $ g $ is the translation on the real line. Then Proposition \ref{prop-no hyers-ulam of translation} is applicable to the translation on $ \R $, which does not have Hyers-Ulam stability. So we may assume that real parabolic M\"obius map does not fix $ \infty $. In this section, we separate the real line to subintervals of which endpoints are $ -\frac{d}{c} $, $ \alpha $ and $ \frac{a}{c} $. Moreover, we  calculate the image of each intervals under parabolic M\"obius map. 
\medskip
\begin{rk}
Real parabolic M\"obius map can be realized as the restriction of the parabolic M\"obius map on the extended real line. Thus real parabolic map is continuous under spherical metric on $ \hat{\R} $. 
\end{rk}

We use the notation $ +\infty $ as the unbounded limit which is greater than any positive number and $ -\infty $ as the unbounded limit which is less than any negative number. The following auxiliary lemma is for later use. 

\bigskip
\begin{lem} \label{lem-auxiliary calculation for const}
Let $ g $ be the parabolic M\"obius map which does not fix $ \infty $, that is, $ g(x) = \frac{ax+b}{cx+d} $ for $ ad-bc=1 $, $ a+d =\pm 2 $ and $ c \neq 0 $. Then the following equations hold 
\begin{align*}
\frac{a+d}{2} = \frac{2}{a+d}, \quad c\alpha + d = \frac{a+d}{2} \quad \textrm{and} \quad
\alpha = \frac{a}{c} - \frac{2}{c(a+d)} .
\end{align*}
\end{lem}

\begin{proof}
Observe that $ \frac{a+d}{2} = \pm 1 $. Thus
\begin{align} 
1 = \left( \frac{a+d}{2} \right)^2 = \frac{a+d}{2} \cdot \frac{2}{a+d} .
\end{align}
Thus $ \frac{a+d}{2} = \frac{2}{a+d} $. Recall that $ \alpha $ is the unique fixed point of the parabolic M\"obius map $ g $ and $ \alpha = \frac{a-d}{2c} $. Then the following equation 
\begin{align*}
\frac{a}{c} - \alpha = \frac{a}{c} - \frac{a-d}{2c} = \frac{a+d}{2c} = \frac{2}{c(a+d)} 
\end{align*}
holds. Hence, $ \alpha = \frac{a}{c} - \frac{2}{c(a+d)} $. 
\end{proof}

Real parabolic M\"obius map has two cases, one of which is $ (a+d)c > 0 $ and the other is $ (a+d)c < 0 $. Lemma \ref{lem-auxiliary calculation for const} implies that $ \alpha < \frac{a}{c} $ if and only if $ (a+d)c > 0 $. We deal with the case $ (a+d)c > 0 $ in the following lemmas. 

\bigskip
\begin{lem} \label{lem-image of interval under real parabolic map}
Let $ g $ be the real parabolic M\"obius map, that is, $ g(x) = \frac{ax+b}{cx+d} $ where $ a, b, c $ and $ d $ are real numbers for $ ad-bc=1 $, $ a+d =\pm 2 $ and $ c \neq 0 $. Suppose that $ (a+d)c > 0 $. Then $ \alpha < x < +\infty $ if and only if $  \alpha < g(x) < \frac{a}{c} $. The inequality $ -\infty < x < -\frac{d}{c} $ holds if and only if $ \frac{a}{c} < g(x) < +\infty $. 
\end{lem}

\begin{proof}
The following equivalent conditions prove the first part of the lemma. Observe that $ \frac{ax+b}{cx+d} = \frac{a}{c} - \frac{1}{c(cx+d)} $. Lemma  \ref{lem-auxiliary calculation for const} is applied to the followings. 
\begin{align*}
+\infty > x > \alpha & \Longleftrightarrow +\infty > c^2x + cd > c^2\alpha + cd \\
& \Longleftrightarrow +\infty > c(cx+d) > c(c\alpha+d) = \frac{c(a+d)}{2} > 0 \\
& \Longleftrightarrow \quad \ 0 < \frac{1}{c(cx+d)} < \frac{2}{c(a+d)} < +\infty \\
& \Longleftrightarrow \quad \ 0 > -\frac{1}{c(cx+d)} > - \frac{2}{c(a+d)} \\
& \Longleftrightarrow \quad \frac{a}{c} > \frac{a}{c} -\frac{1}{c(cx+d)} > \frac{a}{c} - \frac{2}{c(a+d)} = \alpha \\
& \Longleftrightarrow \quad \frac{a}{c} > \frac{ax+b}{cx+d} > \alpha .
%
\intertext{Hence, $ \alpha < x < +\infty $ if and only if $ \alpha < g(x) < \frac{a}{c} $. Moreover, the following equivalent condition completes the proof}
%
-\infty < x < -\frac{d}{c} & \Longleftrightarrow -\infty < c^2x + cd < 0 \\
& \Longleftrightarrow -\infty < \frac{1}{c(cx+d)} < 0 \\
& \Longleftrightarrow \quad \ 0 < -\frac{1}{c(cx+d)} < +\infty \\
& \Longleftrightarrow \quad \frac{a}{c} < \frac{a}{c} -\frac{1}{c(cx+d)} < +\infty \\
& \Longleftrightarrow \quad \frac{a}{c} < \frac{ax+b}{cx+d} < +\infty .
\end{align*}
Hence, $ -\infty < x < -\frac{d}{c} $ if and only if $ \frac{a}{c} < g(x) < +\infty $. 
\end{proof}

\bigskip
\begin{cor} \label{cor-second iterated image under g}
Let $ g $ be the real parabolic M\"obius map, which is $ g(x) = \frac{ax+b}{cx+d} $ where $ a, b, c $ and $ d $ are real numbers for $ ad-bc=1 $, $ a+d =\pm 2 $ and $ c \neq 0 $. Suppose that $ (a+d)c > 0 $. If $ -\infty < x < -\frac{d}{c} $, then $ \alpha < g^2(x) < \frac{a}{c} $. 
\end{cor}

\bigskip
\begin{lem}  \label{lem-limit of image under g}
Let $ g $ be the real parabolic M\"obius map, which is $ g(x) = \frac{ax+b}{cx+d} $ where $ a, b, c $ and $ d $ are real numbers for $ ad-bc=1 $, $ a+d =\pm 2 $ and $ c \neq 0 $. Suppose that $ (a+d)c > 0 $. If $ x $ satisfies the inequality $ \alpha < x < \frac{a}{c} $, then the inequality 
\begin{align*}
\alpha < \cdots < g^n(x) < g^{n-1}(x)< \cdots < g(x) < x
\end{align*}
holds and $ \lim_{n \rightarrow +\infty} g^n(x) = \alpha $. 
\end{lem}

\begin{proof}
The following equivalent conditions hold
\begin{align*}
\alpha < x < \frac{a}{c} \ & \Longleftrightarrow c^2\alpha + cd < c^2x + cd < c^2\frac{a}{c}+cd \\
\ & \Longleftrightarrow c(c\alpha+d) < c(cx+d) < c(a+d) .
\end{align*}
Since $ c(a+d) > 0 $ and $ c(c\alpha + d) = \frac{c(a+d)}{2} $, the inequality $ 0 < c(cx+d) $ is satisfied. Thus $ g(x) - x $ is as follows
\begin{align*}
g(x) - x = \frac{ax+b}{cx+d} - x = - \frac{cx^2 - (a-d)x -b}{cx+d} = - \frac{c(x-\alpha)^2}{cx+d} .
\end{align*}
Thus the equation 
$ g(x) - x = - \frac{c^2(x-\alpha)^2}{c(cx+d)} $
holds. Since both $ c^2 $ and $ c(cx+d) $ are positive number, $ g(x) - x < 0 $ for every $ x $ in the interval $ \left( \alpha, \frac{a}{c} \right) $. Lemma \ref{lem-image of interval under real parabolic map} implies that $ g((\alpha, +\infty)) = \left( \alpha, \frac{a}{c} \right) $. Thus we obtain that $ g\left(\left( \alpha, \frac{a}{c} \right)\right) \subset \left( \alpha, \frac{a}{c} \right) $. By induction, for any $ x_0 \in \left( \alpha, \frac{a}{c} \right) $ the set $ \{ g^k(x_0) \ | \ k \in \N \} $ is contained in the same interval. Denote $ g^n(x_0) $ by $ x_n $ for each $ n \in \N $. Thus the following inequality
\begin{align*}
g(x_{n-1}) - x_{n-1} = g \circ g^{n-1}(x_0) - g^{n-1}(x_0) = g^n(x_0) - g^{n-1}(x_0) < 0 
\end{align*}
holds. Thus inductively we obtain that 
\begin{align*}
\alpha < \cdots < g^n(x) < g^{n-1}(x)< \cdots < g(x) < x .
\end{align*}
The sequence $ \{g^n(x) \}_{n \in \N_0} $ is a decreasing sequence bounded below by $ \alpha $. Then there exists $ \lim_{n \rightarrow +\infty} g^n(x) $, say $ \beta $. However,  by the continuity of $ g $, $ \beta $ is a fixed point of $ g $. The uniqueness of the fixed point of $ g $ implies that $ \beta = \alpha $. Hence, $ \lim_{n \rightarrow +\infty} g^n(x) $ is $ \alpha $. 
\end{proof}

\begin{figure}
    \centering
    \begin{subfigure}[b]{0.48\textwidth}
        \includegraphics[width=\textwidth]{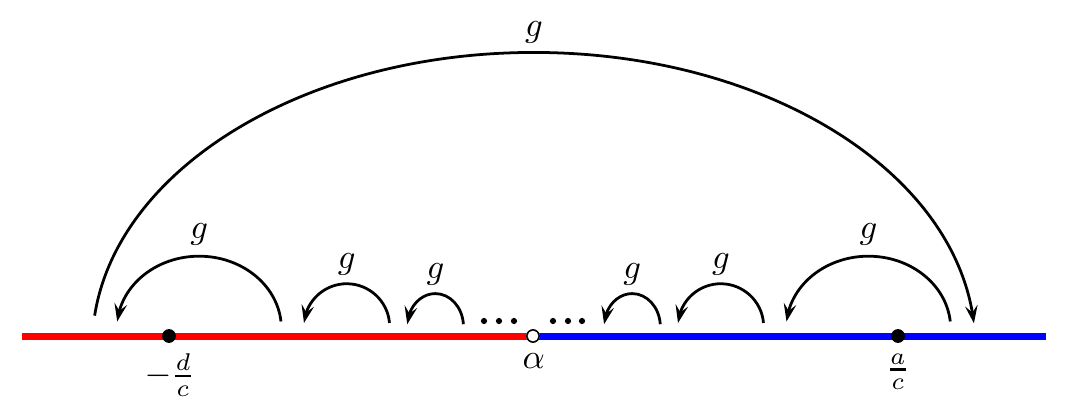}
        \caption{$(a+d)c>0$}
        \label{fig:real parabolic map (a+d)c>0}
    \end{subfigure}
    ~ 
    \begin{subfigure}[b]{0.48\textwidth}
        \includegraphics[width=\textwidth]{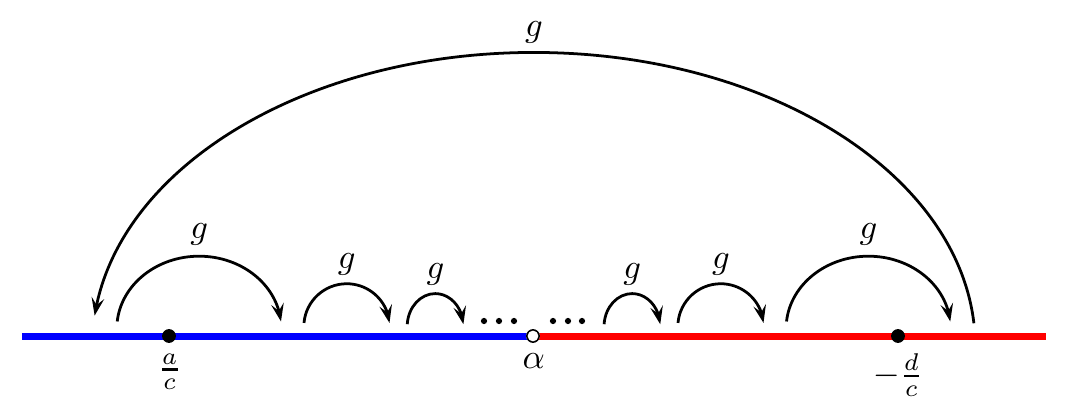}
        \caption{$(a+d)c<0$}
        \label{fig:real parabolic map (a+d)c<0}
    \end{subfigure}
    \caption{Iterated images under real parabolic M\"obius map, $ g(z)=\frac{az+b}{cz+d} $}\label{fig:fig:real parabolic map}
\end{figure}
%
We obtain the similar result if the inequality $ c(a+d) < 0 $ holds in the following lemma.

\bigskip
\begin{lem} \label{lem-image of interval under real parabolic map 2}
Let $ g $ be the real parabolic M\"obius map, which is $ g(x) = \frac{ax+b}{cx+d} $ where $ a, b, c $ and $ d $ are real numbers for $ ad-bc=1 $, $ a+d =\pm 2 $ and $ c \neq 0 $. Suppose that $ (a+d)c < 0 $. Then 
\begin{itemize}
\item $ -\infty < x < \alpha $ if and only if $ \frac{a}{c}<g(x)<\alpha $, 
\item $ -\frac{d}{c} < x < +\infty $ if and only if $ -\infty < g(x) < \frac{a}{c} $ and 
\item if $ \frac{a}{c} < x < \alpha $, then $ x < g(x) < g^2(x) < \cdots < g^n(x) < \cdots < \alpha $ and moreover, $ \lim_{n \rightarrow +\infty} g^n(x) = \alpha $. 
\end{itemize}
\end{lem}

\begin{proof}
The following equivalent conditions prove the first part of the lemma. 
\begin{align*}
-\infty < x < \alpha & \Longleftrightarrow -\infty < c^2x + cd < c^2\alpha + cd \\
& \Longleftrightarrow -\infty < c(cx+d) < c(c\alpha+d) = \frac{c(a+d)}{2} < 0 \\
& \Longleftrightarrow \quad \ 0 > \frac{1}{c(cx+d)} > \frac{2}{c(a+d)} > -\infty \\
& \Longleftrightarrow \quad \ 0 < -\frac{1}{c(cx+d)} < -\frac{2}{c(a+d)} \\
& \Longleftrightarrow \quad \frac{a}{c} < \frac{a}{c} - \frac{1}{c(cx+d)} < \frac{a}{c} - \frac{2}{c(a+d)}  
\end{align*}
Since $ \frac{a}{c} - \frac{1}{c(cx+d)} = \frac{ax+b}{cx+d} $ and $ \frac{a}{c} - \frac{2}{c(a+d)} = \alpha $ by Lemma \ref{lem-auxiliary calculation for const}, the condition $ -\infty < x < \alpha $ is equivalent to $ \frac{a}{c}<g(x)<\alpha $. The proof of the second and third parts is similar to that of Lemma \ref{lem-image of interval under real parabolic map} and Lemma \ref{lem-limit of image under g} as well as the first part of the lemma. Then the remaining proofs are left to the reader. 
\end{proof}

\bigskip
\begin{lem} \label{lem-limit of image under g inverse1}
Let $ g $ be the real parabolic M\"obius map, which is $ g(x) = \frac{ax+b}{cx+d} $ where $ a, b, c $ and $ d $ are real numbers for $ ad-bc=1 $, $ a+d =\pm 2 $ and $ c \neq 0 $. Suppose that $ (a+d)c > 0 $. Then $ -\infty < x < \alpha $ if and only if $  -\frac{d}{c} < g^{-1}(x) < \alpha $. The point $ x $ satisfies the inequality $ -\frac{d}{c} < x < \alpha $, then the inequality 
$$
x< g^{-1}(x) < g^{-2}(x) < \cdots < g^{-n}(x) < \cdots < \alpha
$$
holds and $ \lim_{n \rightarrow +\infty} g^{-n}(x) = \alpha $. 
\end{lem}

\begin{proof}
The straightforward calculation implies that $ g^{-1}(x) = \frac{dx-b}{-cx+a} $, which is also real parabolic M\"obius map. $ \alpha $ is the fixed point of $ g^{-1} $ and $ -c(a+d) < 0 $. Replace $ a,b,c $ and $ d $ of the map $ g $ by $ d,-b,-c $ and $ a $ respectively. Then apply the proof of Lemma \ref{lem-image of interval under real parabolic map 2} to the proof for the map $ x \mapsto \frac{dx-b}{-cx+a} $. It completes the proof. 
\end{proof}

\bigskip
\begin{cor} \label{cor-iterated image of point from fixed point}
Let $ g $ be the real parabolic M\"obius map, which is $ g(x) = \frac{ax+b}{cx+d} $ where $ a, b, c $ and $ d $ are real numbers for $ ad-bc=1 $, $ a+d =\pm 2 $ and $ c \neq 0 $. For every $ x \in \left( -\frac{d}{c}, \alpha \right) $, the number $ g^{N_1}(x) $ is contained in the interval $ \left( -\infty, -\frac{d}{c} \right) $ for some $ N_1 \in \N $. 
\end{cor}

In the case that $ (a+d)c < 0 $, the lemma holds as follows. 

\bigskip
\begin{lem} \label{lem-limit of image under g inverse2}
Let $ g $ be the real parabolic M\"obius map, which is $ g(x) = \frac{ax+b}{cx+d} $ where $ a, b, c $ and $ d $ are real numbers for $ ad-bc=1 $, $ a+d =\pm 2 $ and $ c \neq 0 $. Suppose that $ (a+d)c < 0 $. Then $ -\infty < x < \alpha $ if and only if $  \alpha < g^{-1}(x) < -\frac{d}{c} $. If the inequality $ \alpha < x < \infty $ holds, then the inequality 
$$
\alpha < \cdots < g^{-n}(x) < g^{-(n-1)}(x)< \cdots < g^{-1}(x) < x 
$$
holds and $ \lim_{n \rightarrow +\infty} g^{-n}(x) = \alpha $.  
\end{lem}

\section{Non stability of real parabolic M\"obius map}
In this section, we prove the non stability of the sequence $ \{b_n \}_{n \in \N_0} $ with for any initial point $ b_0 \in \R $ satisfying $ b_{n+1} = g(b_n) $ for every $ n \in \N_0 $. 

\bigskip
\begin{thm}
Let $ g $ be the real parabolic M\"obius map, which is $ g(x) = \frac{ax+b}{cx+d} $ where $ a, b, c $ and $ d $ are real numbers for $ ad-bc=1 $, $ a+d =\pm 2 $ and $ c \neq 0 $. Let $ \{b_n \}_{n \in \N_0} $ be the sequence satisfying $ b_{n+1} = g(b_n) $ for every $ n \in \N_0 $. Then for any $ b_0 \in \R $ and for any given $ \e > 0 $, the sequence $ \{b_n \}_{n \in \N_0} $ has no Hyers-Ulam stability. 
\end{thm}

\begin{proof}
Assume first that $ (a+d)c > 0 $. The proof for the case that $ (a+d)c < 0 $ is similar. Corollary \ref{cor-second iterated image under g} and Corollary \ref{cor-iterated image of point from fixed point} imply that for each $ x \in \R $ there exists $ N \in \N $ such that $ g^N(x) \in \left( \alpha, \frac{a}{c} \right) $. Lemma \ref{lem-limit of image under g} and Lemma \ref{lem-limit of image under g inverse1} imply that $ \lim_{n \rightarrow \pm \infty} g^n(x) = \alpha $. Using same method in the proof of Proposition \ref{prop-no hyers-ulam of parabolic mobius}, we show that there exists a sequence $ \{a_n \}_{n \in \N_0} $ which satisfies the following properties 
\begin{enumerate}
\item $ | a_{n+1} - g(a_n) | \leq \e $ for all $ n \in \N_0 $, 
\item the sequence $ \{ a_n\}_{n \in \N_0} $ is pre-periodic one, and 
\item $ | a_{N_0} - b_{N_0} | \geq 1 + \e $ for some big enough $ N_0 \in \N $. 
\end{enumerate}
Then $ | a_n - b_n | \geq 1 $ for infinitely many $ n \in \N $ due to the periodicity of the sequence $ \{ a_n\}_{n \in \N_0} $. Hence, $ \{ b_n\}_{n \in \N_0} $ has no Hyers-Ulam stability. We construct the sequence $ \{ a_n\}_{n \in \N_0} $ satisfying the above properties. 
\smallskip \\
For every $ x \in \R $, for some $ N_0 \in \N $, $ g^{N_0}(x) $ is contained in the interval $ \left( \alpha, \alpha + \frac{\e}{2} \right) $ by Lemma \ref{lem-limit of image under g}. Observe that the distance between any point $ x' \in \left( \alpha - \frac{\e}{2}, \alpha \right) $ and $ g^{N_0}(x) $ is less than $ \e $. Moreover, the point $ x' $ satisfies that $ g^{N_1}(x') \in \left(-\infty, -\frac{d}{c} \right) $ by Corollary \ref{cor-iterated image of point from fixed point}. Since $ g^{N_1}(x') $ can be an arbitrary point in the interval $ \left(-\infty, -\frac{d}{c} \right) $ and the distance between $ \alpha $ and $ -\frac{d}{c} $ is $ \frac{1}{|c|} $, we may choose the point $ q $ such that $ q \in \left( \alpha - \frac{\e}{2}, \alpha \right) $ and $ | g^{N_1}(q) - \alpha | > \max \left\{ \frac{1}{|c|},\ 1+\e \right\} $. Thus $ g^{N_1+N_2}(q) \in \left( \alpha, \alpha + \frac{\e}{2} \right) $ for some $ N_2 \in \N $. 
\medskip \\
Define the pre-periodic sequence $ \{ a_n\}_{n \in \N_0} $ as follows 
$$
\{ a_0, g(a_0),g^2(a_0), \ldots ,g^{N_0}(a_0), q, g(q), \ldots,g^{N_1}(q), \ldots, g^{N_1+N_2}(q), q ,g(q), \ldots \}  .
$$
Thus $ |a_{n+1} - g(a_n) | \leq \e $ for all $ n \in \N_0 $ and $ a_{N_1+N_2+k} = a_{k+1} $ for every $ k \geq N_0 $. Furthermore, since $ g^{N_1}(q) = a_{N_0+N_1+2} $ and $ b_n \in \left( \alpha, \alpha + \frac{\e}{2} \right) $ for all big enough $ n \geq N $, the inequality 
\begin{align*}
| a_{N_0+N_1+2 + k(N_1 + N_2)} - b_{N_0+N_1+2} | &= | a_{N_0+N_1+2} - b_{N_0+N_1+2} |  \\
& \geq | a_{N_0+N_1+2} - \alpha | - | \alpha - b_{N_0+N_1+2} | \\
& \geq 1+\e - \frac{\e}{2} \\
& > 1
\end{align*}
for all $ k \geq N_0 + N $. Then $ |b_n - a_n| > 1 $ for infinitely many $ n \in \N $. Hence, the sequence $ \{ b_n\}_{n \in \N_0} $ does not have Hyers-Ulam stability. 
\end{proof}


\end{document}